\documentclass[amssymb,twoside,11pt]{article}
\usepackage{latexsym,amsmath,graphicx}
\usepackage{amsfonts}
\usepackage{enumitem}
\newtheorem{theorem}{Theorem}[section]
\newtheorem{Lemma}[theorem]{{\bf Lemma}}

\newtheorem{definition}{Definition}[section]
\numberwithin{equation}{section}
\newenvironment{proof}{\indent{\em Proof:}}{\quad \hfill
$\Box$\vspace*{2ex}}

\begin{document}
\setcounter{page}{1}
\begin{center}
\vspace{0.4cm} {\large{\bf On Nonlinear  Hybrid  Fractional Differential Equations with Atangana-Baleanu-Caputo Derivative }} \\
\vspace{0.5cm}
Sagar T. Sutar  $^{1}$\\
sutar.sagar007@gmail.com\\
\vspace{0.35cm}
Kishor D. Kucche $^{2}$ \\
kdkucche@gmail.com \\

\vspace{0.35cm}
$^1$ Department of Mathematics, Vivekanand College (Autonomous), Kolhapur-416003, Maharashtra, India.\\
$^{2}$ Department of Mathematics, Shivaji University, Kolhapur-416 004, Maharashtra, India.\\
\end{center}

\def\baselinestretch{1.0}\small\normalsize

\begin{abstract}
In this paper, we develop the theory of nonlinear hybrid fractional differential equations involving Atangana--Baleanu--Caputo (ABC) fractional derivative.  We construct the equivalent fractional integral equation and establish the existence results through it. Further, we build up the theory of inequalities for ABC--hybrid fractional differential equations and use it to examine the uniqueness, existence of a maximal and minimal solution and the comparison results.  
\end{abstract}
\noindent\textbf{Key words:} Hybrid fractional differential equations,  Atangana--Baleanu--Caputo fractional derivative, Existence and uniqueness, Fractional differential inequalities, Comparison results, Maximal and minimal solutions.\\ \\
\noindent
\textbf{2010 Mathematics Subject Classification:} {26A33, 34A12, 34A40, 35B50, 34A99}
\def\baselinestretch{1.5} 
\allowdisplaybreaks
\section{Introduction}
Lakshmikantham and Vatsala \cite{Laksh,Laksh2} developed the primary theory of fractional differential inequalities with Reimann-Liouville and Caputo fractional derivative. Authors utilized the investigated fractional inequalities and the comparison results to study the existence of local, extremal and global solutions to nonlinear fractional differential equations (FDEs). Dhage and Lakshmikantham \cite{dhage} initiated the study of  first order hybrid differential equations and investigated the basic results pertaining to existence and uniqueness of solution. Further, differential inequalities obtained in connection with hybrid FDEs   utilized   to examine   comparison results and qualitative properties of solution. Adopting the similar approach of \cite{dhage},  Zhao et al. in \cite{yzh} extended the  study of  first order hybrid differential equations to hybrid FDEs involving Riemann-Liouville fractional derivative. Further, different class of Hybrid FDEs subject to various initial and boundary conditions have also been studied by several researchers \cite{BashirAHMAD,SFER,JCaba,Ssitho,Sunz}.
 
  On the other hand, intending to eliminate the singular kernel in traditional fractional derivatives, Caputo and Fabrizio \cite{CF1} presented a fractional derivative  with the exponential kernel and  Atangana--Baleanu \cite{ABC1} introduced a fractional derivative in the sense of Caputo with Mittag--Leffler function as its kernel, which notable as  ABC--fractional derivative.  The advantage of  ABC--fractional derivative is that it is nonlocal and has a non-singular kernel. Because of which it has numerous applications in demonstrating different  problems that includes different diseases, such as,  dengue fever outbreak \cite{Amin1}, tumor-immune surveillance mechanism \cite{ABC2}, the clinical implications of diabetes and tuberculosis coexistence \cite{CF5}, the free motion of a coupled oscillator \cite{Amin2},  smoking models \cite{Sume} and coronavirus \cite{MSA}. For the 
 fundamental development  in the theory nonlinear  ABC-FDEs, we refer the reader to the work of Jarad et al. \cite{Jarad}, Baleanu et al.\cite{Bale4}, Syam et al.\cite{Syam}, Afshari et al. \cite{Af},  Shah et al. \cite{shah} and  Ravichandran et al.  \cite{Ravi1,Ravi2,Ravi3,Ravi4}. 
 
 Motivated by the works of \cite{dhage,yzh} and in continuation of a past work we have done in \cite{KKSS}, we develop the theory of nonlinear hybrid ABC-FDEs of the form
  \begin{align}
  ^{ABC}{_0\mathcal{D}^\alpha_\tau}\left( \dfrac{\omega(\tau)}{f(\tau,\omega(\tau))}\right) &=g(\tau,\omega(\tau)), \;\text{a.e.}~\tau\in J,\label{He1.1}\\
  \omega(0)&=\omega_0,\label{He1.2}
  \end{align}
 where,
 \begin{itemize} [topsep=0pt,itemsep=-1ex,partopsep=1ex,parsep=1ex]
 \item[(i)]  $J=[0,T],\;T>0$ and  $0<\alpha<1$,
\item[(ii)] $^{ABC}{_0\mathcal{D}}_{\tau}^\alpha$ denotes left ABC-  fractional differential operator of order $\alpha$ with lower terminal $ 0$,
 \item[(iii)] $f\in C(J\times\mathbb{R},\mathbb{R}$$\backslash\left\lbrace 0\right\rbrace )$, $\omega\in C(J)\;\text{\normalfont {and}} \;^{ABC}{_0\mathcal{D}}_{\tau}^\alpha h \in C(J),$ where $h(\tau)=\dfrac{\omega(\tau)}{f(\tau,\omega(\tau))},\tau\in J$.
 \item[(iv)] $g\in \mathcal{C}$ is such that $g(0,\omega(0))=0$, where 
 $$\mathcal{C}=\left\lbrace h ~|~h:J\times\mathbb{R}\to \mathbb{R}  ~\text{is continious}, ~h(\tau,\cdot)\; \text{is measurable and } h(\cdot,\omega) \;\text{is continious} \right\rbrace.
 $$
 \end{itemize}
The primary aim of the current study is to determine the  equivalent fractional integral equation to ABC--hybrid FDEs  \eqref{He1.1}-\eqref{He1.2} and explore the existence results. Further, we build up the theory of inequalities for ABC--hybrid FDEs and use it to examine the existence of a maximal and minimal solution and the comparison results.

The current paper is coordinated as follows. In section 2,  we review essential definitions and results about ABC-fractional derivative. In section 3, we give equivalent fractional integral equations and derive existence result through it.  In section 4, we acquire fractional differential inequalities for ABC--hybrid FDEs.  Section 5 deals with the existence of maximal and minimal solutions of ABC--hybrid FDEs.  In section 6, we determine comparison results relating to ABC--hybrid FDEs.
 
 \section{Preliminaries} 
 In this section, we recall the basic definitions and the results about ABC-fractional derivative which will be used later.
\begin{definition}
 A function $\omega\in AC(J,\mathbb{R})$ is said to be solution of  {\normalfont ABC-hybrid-FDEs} \eqref{He1.1}--\eqref{He1.2}, if the mapping  $u\to\dfrac{u}{f(\tau,u
  )}$ is absolutely continuous for each $u\in \mathbb{R}$ and $\omega$ satisfies {\normalfont  ABC-hybrid-FDEs} \eqref{He1.1}--\eqref{He1.2}, where
  $ 
  AC(J,\mathbb{R})=\left\lbrace h~|~h:J\to \mathbb{R} \;\text{is absolutely continious }\right\rbrace $.  
\end{definition}
\begin{definition} {\normalfont\cite{Syam}}
Let $p\in[1,\infty)$ and $\Omega$ be an open subset of $\mathbb{R}$ the Sobolev space $H^p(\Omega)$ is defined as 
$$ H^p(\Omega)=\left\lbrace f\in L^2(\Omega):D^\beta f\in L^2(\Omega), \text{for all } |\beta|\leq p\right\rbrace .$$
\end{definition}
\begin{definition}{\normalfont\cite{ABC1}}
Let $\omega\in H^1(0,T)$ and $\alpha \in [0,1]$, the left Atangana--Baleanu--Caputo fractional derivative of $\omega$ of order $\alpha $ is defined  by
\begin{equation*}
^{ABC}{_0}\mathcal{D}_{\tau}^\alpha \omega(\tau)=\dfrac{B(\alpha)}{1-\alpha}\int_{0}^{\tau}\mathbb{E}_\alpha\left[ -\dfrac{\alpha}{1-\alpha}(\tau-\sigma)^\alpha\right] \omega'(\sigma)d\sigma,
\end{equation*} 
where $B(\alpha)>0$ is a normalization function satisfying $B(0)=B(1)=1$ and $\mathbb{E}_\alpha$ is one parameter Mittag-Leffler function  {\normalfont\cite{Kilbas,Diethelm}}  defined by 
$$
\mathbb{E}_\alpha (z)=\sum_{n=0}^{n=\infty}\frac{z^n}{\Gamma(n \alpha+1)}.
$$
The associated fractional integral is defined  by
\begin{equation*}
^{AB}{_0}I_\tau^\alpha \omega(\tau)=\dfrac{1-\alpha}{B(\alpha)} \omega(\tau)+\dfrac{\alpha}{B(\alpha)}~{_0}I_\tau^\alpha \omega(\tau).
\end{equation*} 
where 
$${_0}I_\tau^\alpha \omega(\tau)= \dfrac{1}{\Gamma(\alpha)}\int_{0}^{\tau}(\tau-\sigma)^{\alpha-1}\omega(\sigma)d\sigma,
$$
is the Riemann--Liouville fractional integral {\normalfont\cite{Kilbas,Diethelm}} of $\omega$ of order $\alpha$.
\end{definition} 
\begin{Lemma}{\normalfont\cite{Bal1}} If $0<\alpha<1$,~ then 
$
^{AB}{_0}I_\tau^\alpha\; \left( ^{ABC}{_0}\mathcal{D}_{\tau}^\alpha \omega(\tau)\right) =\omega(\tau)-\omega(0).
$
\end{Lemma}

\begin{Lemma}{\normalfont\cite{Syam,Ravi1}}\label{ABCLm4.1}
The equivalent fractional integral equation to the { \normalfont the  ABC-FDEs }
\begin{align*}
^{ABC}{_0\mathcal{D}}_\tau^\alpha \omega(\tau)&= f\left( \tau, \omega(\tau)\right), \tau\in J=[0, T],\; T>0,  \\
\omega(0)&=\omega_0,
\end{align*}
is given by
\begin{equation*}
\omega(\tau)=\omega_0+\dfrac{1-\alpha}{B(\alpha)}  f(\tau,\omega(\tau)) +\dfrac{\alpha}{B(\alpha)\Gamma(\alpha)}\int_{0}^\tau(\tau-\sigma)^{\alpha-1}f(\sigma,\omega(\sigma))d\sigma.
\end{equation*}
\end{Lemma}
 
\begin{definition}{\normalfont\cite{pra,kilb,Erd}}
The generalized Mittag-Leffler function $ \mathbb{E}_{\alpha,\beta}^\gamma(z) $ for the complex numbers $\alpha,\beta,\gamma $ with Re($\alpha)>0$ is defined as 
$$
 \mathbb{E}_{\alpha,\beta}^\gamma(z)=\sum_{k=0}^{\infty}\dfrac{(\gamma)_k}{\Gamma(\alpha k+\beta)}\dfrac{z^k}{k!},
$$
where $(\gamma)_k$ is the Pochhammer symbol given by
$$
(\gamma)_0=1,\;(\gamma)_k=\gamma(\gamma+1)\cdots(\gamma+k-1),\; k=1,2,\cdots
$$
\end{definition}
Note that,
$$
\mathbb{E}_{\alpha,\beta}^1(z) =\mathbb{E}_{\alpha,\beta}(z) ~\mbox{and }~\mathbb{E}_{\alpha,1}^1(z) =\mathbb{E}_{\alpha}(z).
$$
\begin{Lemma}{\normalfont\cite{Bal1}}\label{ABCLm2.2} 
Let  $0<\alpha<1$ and $\beta,\sigma,\lambda\in\mathbb{C} \left( Re(\beta)>0 \right)$. Then
$$
^{ABC}{_0\mathcal{D}}_{\tau}^\alpha \left[\tau^{\beta-1}\;\mathbb{E}_{\alpha,\, \beta}^\sigma \, (\lambda \, \tau^\alpha) \right]=\dfrac{B(\alpha)}{1-\alpha} \; \tau^{\beta-1}\;\mathbb{E}_{\alpha,\,\beta}^{1+\sigma}(\lambda \,\tau^\alpha).
$$
\end{Lemma}
\begin{Lemma}{\normalfont\cite{KKSS}}\label{Lm2.4}
If $m$ is any differentiable function  on $J$ such that $^{ABC}{_0\mathcal{D}}_{\tau}^\alpha m\in C(J)$  and there exists $\tau_0\in (0,T]$ with $m(\tau_0)= 0,\;m(\tau)\leq 0,\; \tau\in [0,\tau_0)$,  then ~$^{ABC}{_0\mathcal{D}}_{\tau}^\alpha m(\tau_0)\geq 0$.
\end{Lemma}

\begin{Lemma}\label{LM2.3}{\normalfont\cite{dhage}}
Let $\mathcal{S}$ be a non-empty, closed convex and bounded subset of Banach algebra $\Omega$ and let $\mathcal{F}_1:\Omega\to \Omega$ and $\mathcal{F}_2:\mathcal{S}\to \Omega$ be two operators such that
 \begin{itemize} [topsep=0pt,itemsep=-1ex,partopsep=1ex,parsep=1ex]
\item [\normalfont(i)] $\mathcal{F}_1$ is Lipschitzian   with Lipschitz constant $\alpha$,
\item [\normalfont(ii)] $\mathcal{F}_2$ is completely continious,
\item [\normalfont(iii)]$\omega=\mathcal{F}_1\omega \mathcal{F}_2\eta\implies \omega\in \mathcal{S}$ for all $\eta\in \mathcal{S}$, and
\item [\normalfont(iv)]$\alpha M<1$, where $M=\sup\left\lbrace ||\mathcal{F}_2(\omega)||:\omega\in \mathcal{S}\right\rbrace $,
\end{itemize}
then the operator $ \mathcal{F}_1\omega \mathcal{F}_2\omega=\omega$ has a solution in $\mathcal{S}$.
\end{Lemma}
\section{Existence result}

In the following Theorem, we derive an equivalent fractional integral equation to ABC--hybrid FDEs  \eqref{He1.1}-\eqref{He1.2}. 
\begin{theorem}\label{HYThm3.1}
Let $g\in  C(J\times\mathbb{R},\mathbb{R})$ and assume that, $\omega\to\dfrac{\omega}{f(\tau,\omega)}$ is increasing in $\mathbb{R}$ a.e. for each $\tau\in J$.  Then  $\omega\in AC(J,\mathbb{R})$ is a solution of  {\normalfont ABC--hybrid FDEs} \eqref{He1.1}--\eqref{He1.2} if and only if $\omega$ is a solution of fractional integral equation
\begin{small}
\begin{equation}\label{He3.1}
\omega(\tau)=f(\tau,\omega(\tau))\left[  \dfrac{\omega_0}{f(0,\omega_0)}+\dfrac{1-\alpha}{B(\alpha)} g(\tau,\omega(\tau))+\dfrac{\alpha}{B(\alpha)(1-\alpha)}\int_{0}^{\tau}(\tau-\sigma)^{\alpha-1}g(\sigma,\omega(\sigma))d\sigma\right],\tau\in J.
\end{equation}
 \end{small} 
\end{theorem}
\begin{proof}
In the view of Lemma \ref{ABCLm4.1}, if $\omega$ is a solution of  ABC--hybrid FDEs  \eqref{He1.1}--\eqref{He1.2}, then $\omega$ satisfies fractional integral equation
\begin{equation}\label{He3.2}
\begin{small}
\dfrac{\omega(\tau)}{f(\tau,\omega(\tau))}=\dfrac{\omega_0}{f(0,\omega_0)}+\dfrac{1-\alpha}{B(\alpha)} g(\tau,\omega(\tau))+\dfrac{\alpha}{B(\alpha)(1-\alpha)}\int_{0}^{\tau}(\tau-\sigma)^{\alpha-1}g(\sigma,\omega(\sigma))d\sigma,\;\tau\in J,
\end{small}
\end{equation}

\noindent which gives Eq.\eqref{He3.1}. Conversely, let $\omega$ satisfies Eq.\eqref{He3.1}. Then it can be written in the form of Eq.\eqref{He3.2}. 
Operating $^{ABR}{_0\mathcal{D}^\alpha_\tau}$ on both sides of Eq.\eqref{He3.2}, we obtain
\begin{align*}
^{ABR}{_0\mathcal{D}^\alpha_\tau} \left( \dfrac{\omega(\tau)}{f(\tau,\omega(\tau))}\right) &=\;^{ABR}{_0\mathcal{D}^\alpha_\tau}\left[\dfrac{\omega_0}{f(0,\omega_0)}+^{AB}{_0}I_\tau^\alpha g(\tau,\omega(\tau)) \right]\\
& =\dfrac{\omega_0}{f(0,\omega_0)}^{ABR}{_0\mathcal{D}^\alpha_\tau}(1)+g(\tau,\omega(\tau))\\
&=\dfrac{\omega_0}{f(0,\omega_0)}\mathbb{E}_\alpha\left[ -\dfrac{\alpha}{1-\alpha}\tau^\alpha\right]+g(\tau,\omega(\tau),\;\tau\in J.
\end{align*}
This gives, 
$$
^{ABR}{_0\mathcal{D}^\alpha_\tau} \left( \dfrac{\omega(\tau)}{f(\tau,\omega(\tau))}\right) -\dfrac{\omega_0}{f(0,\omega_0)}\mathbb{E}_\alpha\left[ -\dfrac{\alpha}{1-\alpha}\tau^\alpha\right]=g(\tau,\omega(\tau)),\;\tau\in J.
$$
Using relation between fractional differential operators $^{ABR}{_0\mathcal{D}^\alpha_\tau} $ and $^{ABC}{_0\mathcal{D}^\alpha_\tau} $ given in Theorem 1\cite{ABC1}, we obtain
$$^{ABC}{_0\mathcal{D}^\alpha_\tau} \left( \dfrac{\omega(\tau)}{f(\tau,\omega(\tau))}\right) =g(\tau,\omega(\tau)),\; \tau\in J. $$
Now putting  $\tau=0$ in Eq.\eqref{He3.2} and using the fact $g(0,\omega(0))=0$, we obtain
\begin{equation}\label{Hye3.3}
\dfrac{\omega(0)}{f(0,\omega(0))}=\dfrac{\omega_0}{f(0,\omega_0)}.
\end{equation}

\noindent For each $\tau\in J$, consider the mapping $h_\tau:\mathbb{R}\to\mathbb{R}$ defined by,
$$h_\tau(\eta)=\dfrac{\eta}{f(\tau,\eta)},\;\eta\in\mathbb{R}.$$
By assumption $h_\tau:\mathbb{R}\to\mathbb{R}$ is increasing   and hence it is injective. Using definition of $h_\tau$,  Eq.\eqref{Hye3.3}
can be written as $$ h_0(\omega(0))=h_0(\omega_0).$$
Since $h_0$ is injective, we have $\omega(0)=\omega_0$. 
This completes the proof of the Theorem.
\end{proof}

 To prove existence results for solution of ABC--hybrid FDEs \eqref{He1.1}-\eqref{He1.2}, we need following assumptions on $f$ and $g$.
\begin{itemize}
\item [(H1)] The function $f\in C\left( J\times\mathbb{R}, \mathbb{R}\backslash\left\lbrace 0\right\rbrace\right)$ is such 
that, 
\begin{itemize}
\item [(i)] $|f(\tau,\omega)-f(\tau,\eta)|\leq L|\omega-\eta|,\;L>0$,
\item [(ii)] the mapping $\omega\to\dfrac{\omega}{f(\tau,\omega)}$ is increasing in $\mathbb{R}$ a.e. for each $\tau\in J$.
\end{itemize}  
\item [(H2)] The function $g\in \mathcal{C}$ is such that,
$|g(\tau,\omega(\tau)|\leq h(\tau),a.e.\;\tau\in J,\;h\in C(J,\mathbb{R}^+)$.
 \end{itemize}

\begin{theorem}\label{Hythm3.2}
Suppose  the hypotheses {\normalfont(H1)--\normalfont(H2)} hold. Then  {\normalfont ABC--hybrid FDEs} \eqref{He1.1}--\eqref{He1.2} has a solution if  \begin{equation}\label{b4.4}
L\left( \left| \frac{\omega_0}{f(0,\omega_0)}\right|  +\left[ 1-\alpha+\frac{T^\alpha}{1-\alpha}\right]\frac{\|h\|}{B(\alpha)}\right) <1.
\end{equation} 
\end{theorem}
\begin{proof} Let $\Omega=\left( C(J,\mathbb{R}),\|\cdot\|\right) $,
 where $\|\omega\|=\underset{\tau\in J}{\sup}|\omega(\tau)|.$ Then $\Omega$ is Banach algebra with multiplication  defined by 
$$(\omega\eta)\tau=\omega(\tau)\eta(\tau),\; \omega,\eta\in\Omega,\;\tau\in J.$$
Define, \begin{small}
\begin{equation}
R=\dfrac {M_fL\left( \left| \dfrac{\omega_0}{f(0,\omega_0)}\right|  +\left[ 1-\alpha+\frac{T^\alpha}{1-\alpha}\right]\dfrac{\|h\|}{B(\alpha)}\right) }{1-L\left( \left| \dfrac{\omega_0}{f(0,\omega_0)}\right|  +\left[ 1-\alpha+\frac{T^\alpha}{1-\alpha}\right]\dfrac{\|h\|}{B(\alpha)}\right)},
\end{equation}
\end{small}  where $M_f=\underset{\tau\in J}{\sup}|f(\tau,0)|$. In the view condition \eqref{b4.4}, $R>0$. 

\noindent Consider the set,
$$ \mathcal{S}=\left\lbrace \omega\in \Omega:\;\|\omega\|\leq R\right\rbrace.$$
One can verify that $\mathcal{S}$ is closed, convex and bounded subset of Banach algebra $\Omega$. Consider the operators $\mathcal{F}_1:\Omega\to \Omega$ and $\mathcal{F}_1:\mathcal{S}\to \Omega$ defined by,
\begin{align}
(\mathcal{F}_1\omega)(\tau)&=f(\tau,\omega(\tau)),\;\tau\in J,\label{e4.6}\\
(\mathcal{F}_2\omega)(\tau)&=\dfrac{\omega_0}{f(0,\omega_0)}+\dfrac{1-\alpha}{B(\alpha)}  g(\tau,\omega(\tau))+\dfrac{\alpha}{B(\alpha)(1-\alpha)}\int_{0}^{\tau}(\tau-\sigma)^{\alpha-1}g(\sigma,\omega(\sigma))d\sigma,\;\tau\in J,\label{e4.7}
\end{align}
The equivalent fraction integral Eq.\eqref{He3.1} to the ABC--hybrid FDEs \eqref{He1.1}--\eqref{He1.2} can be written in operator equation  form given by 
$$ \omega=\mathcal{F}_1\omega\mathcal{F}_2\omega,\;\omega\in \Omega.$$
We prove  that the operators $\mathcal{F}_1$ and $\mathcal{F}_2$ satisfies conditions of Lemma \ref{LM2.3}. The proof of the same have been given in following steps.

\noindent Step 1) $\mathcal{F}_1$ is Lipschitz.

\noindent Using Lipschitz condition on $f$, for any $\omega,\eta\in \Omega$ and $\tau\in J$ we obtain,
\begin{align*}
|(\mathcal{F}_1\omega)(\tau)-(\mathcal{F}_1\eta)(\tau)|&=|f(\tau,\omega(\tau))-f(\tau,\eta(\tau))|\leq L|\omega(\tau)-\eta(\tau)|,
\end{align*}
This gives, 
$$\|\mathcal{F}_1\omega-\mathcal{F}_1\eta\|\leq L\|\omega-\eta\|,\;\omega,\eta\in \Omega. $$
\noindent Step 2) $ \mathcal{F}_2$ is completely continuous.

\noindent We show that $\mathcal{F}_2:\mathcal{S}\to \Omega$ is a compact and continuous operator on $\mathcal{S}$ into $\Omega$. First we show that $\mathcal{F}_2$ is continuous on $\mathcal{S}$. Let $\left\lbrace \omega_n\right\rbrace $
be a sequence in $\mathcal{S}$ converging to a point $\omega \in S$. Then by the Lebesgue dominated convergence theorem, 
\begin{align*}
\lim\limits_{n\to\infty}
(\mathcal{F}_2\omega_n)(\tau)&=\lim\limits_{n\to\infty}\left[ \dfrac{1-\alpha}{B(\alpha)}g(\tau,\omega_n(\tau))+\dfrac{\alpha}{B(\alpha)(1-\alpha)}\int_{0}^{\tau}(\tau-\sigma)^{\alpha-1}g(\sigma,\omega_n(\sigma))d\sigma\right] \\
&=\dfrac{1-\alpha}{B(\alpha)}\lim\limits_{n\to\infty}g(\tau,\omega_n(\tau))+\dfrac{\alpha}{B(\alpha)(1-\alpha)}\int_{0}^{\tau}(\tau-\sigma)^{\alpha-1}\left\lbrace \lim\limits_{n\to\infty}g(\sigma,\omega_n(\sigma))\right\rbrace d\sigma\\
&=\dfrac{1-\alpha}{B(\alpha)}g(\tau,\omega(\tau))+\dfrac{\alpha}{B(\alpha)(1-\alpha)}\int_{0}^{\tau}(\tau-\sigma)^{\alpha-1}g(\sigma,\omega(\sigma))d\sigma\\
&=(\mathcal{F}_2\omega)(\tau),
 \end{align*}
 for all $\tau \in J$. This shows that $\mathcal{F}_2$ is a continuous operator on $\mathcal{S}$.
 Using hypothesis (H2), for any $\omega\in \mathcal{S}$ and $\tau\in J$, we have
\begin{align}\label{Hye4.7}
|(\mathcal{F}_2\omega)(\tau)|&\leq \left| \dfrac{\omega_0}{f(0,\omega_0)}\right|  +\dfrac{1-\alpha}{B(\alpha)}\left|   g(\tau,\omega(\tau))\right| +\dfrac{\alpha}{B(\alpha)(1-\alpha)}\int_{0}^{\tau}(\tau-\sigma)^{\alpha-1}\left| g(\sigma,\omega(\sigma))\right| d\sigma\nonumber\\
&\leq \left| \dfrac{\omega_0}{f(0,\omega_0)}\right|  +\dfrac{1-\alpha}{B(\alpha)}|h(\tau)| + \dfrac{\alpha}{B(\alpha)(1-\alpha)}\int_{0}^{\tau}(\tau-\sigma)^{\alpha-1}|h(\sigma)| d\sigma\nonumber\\
&\leq \left| \dfrac{\omega_0}{f(0,\omega_0)}\right|  +\dfrac{1-\alpha}{B(\alpha)}\|h\| + \dfrac{\|h\|}{B(\alpha)(1-\alpha)}\tau^\alpha,\nonumber
\end{align}
This gives,
\begin{equation}\label{HYe4.9}
|(\mathcal{F}_2\omega)(\tau)|\leq\left| \frac{\omega_0}{f(0,\omega_0)}\right|  +\left( 1-\alpha+\frac{T^\alpha}{1-\alpha}\right)\frac{\|h\|}{B(\alpha)},\;\omega\in \mathcal{S},\;\tau\in J,
\end{equation}
 which shows that $\mathcal{F}_2$ is uniformly bounded on $J$. Next we prove that $\mathcal{F}_2(\mathcal{S})$ is equicontinious set in $\Omega$. Let any $\omega\in \mathcal{S}$ and   $0\leq \tau_1<\tau_2\leq T$.  Then  we have
\begin{align}\label{ie4.9}
&|\mathcal{F}_2\omega(\tau_1)-\mathcal{F}_2\omega(\tau_2)|
\leq \dfrac{1-\alpha}{B(\alpha)}\left|   g(\tau_1,\omega(\tau_1))-g(\tau_2,\omega(\tau_2))\right|
\nonumber\\&+\dfrac{\alpha}{B(\alpha)(1-\alpha)}\left| \int_{0}^{\tau_1} (\tau_1-\sigma)^{\alpha-1} g(\sigma,\omega(\sigma)) d\sigma-\int_{\tau_1}^{\tau_2} (\tau_2-\sigma)^{\alpha-1} g(\sigma,\omega(\sigma)) d\sigma\right|
\end{align}
Since $g(\tau,\omega)$ is continuous on compact set $J\times [-R,\; R]$, it is uniformly continuous there and hence we have
\begin{equation}\label{Hyee4.10}
\left|   g(\tau_1,\omega(\tau_1))-g(\tau_2,\omega(\tau_2))\right|\to 0,\; \text{as}\; |\tau_1-\tau_2|\to 0,\;\text{for each}\; \omega\in \mathcal{S}.
\end{equation}
Next using hypothesis (H2), we have
\begin{align}\label{HHe4.11}
&\left| \int_{0}^{\tau_1} (\tau_1-\sigma)^{\alpha-1} g(\sigma,\omega(\sigma)) d\sigma-\int_{\tau_1}^{\tau_2} (\tau_2-\sigma)^{\alpha-1}g(\sigma,\omega(\sigma)) d\sigma\right|\nonumber \\
&\leq \int_{0}^{\tau_1}\left\lbrace  (\tau_1-\sigma)^{\alpha-1}-(\tau_2-\sigma)^{\alpha-1}\right\rbrace |g(\sigma,\omega(\sigma))| d\sigma+\int_{\tau_1}^{\tau_2} (\tau_2-\sigma)^{\alpha-1}|g(\sigma,\omega(\sigma))| d\sigma\nonumber\\
&\leq \int_{0}^{\tau_1}\left\lbrace  (\tau_1-\sigma)^{\alpha-1}-(\tau_2-\sigma)^{\alpha-1}\right\rbrace |h(\sigma)| d\sigma+\int_{\tau_1}^{\tau_2} (\tau_2-\sigma)^{\alpha-1}|h(\sigma)| d\sigma\nonumber\\
&\leq \|h\|\left( \int_{0}^{\tau_1}\left\lbrace  (\tau_1-\sigma)^{\alpha-1}-(\tau_2-\sigma)^{\alpha-1}\right\rbrace  d\sigma+\int_{\tau_1}^{\tau_2} (\tau_2-\sigma)^{\alpha-1} d\sigma\right)\nonumber \\
&\leq\|h\|\left( \left\lbrace  \tau_1^\alpha+(\tau_2-\tau_1)^\alpha-\tau_2^\alpha\right\rbrace 
+(\tau_2-\tau_1)^\alpha\right)\nonumber\\
&\leq2\|h\| (\tau_2-\tau_1)^\alpha.
\end{align}
Therefore,
\begin{equation}\label{Hye4.11}
\left| \int_{0}^{\tau_1} (\tau_1-\sigma)^{\alpha-1} g(\sigma,\omega(\sigma)) d\sigma-\int_{\tau_1}^{\tau_2} (\tau_2-\sigma)^{\alpha-1}g(\sigma,\omega(\sigma)) d\sigma\right|\to 0,\;\text{as}\; |\tau_1-\tau_2|\to 0,\;\omega\in \mathcal{S}.
\end{equation}
Therefore it follows from \eqref{ie4.9}, \eqref{Hyee4.10} and \eqref{Hye4.11} that
$$|(\mathcal{F}_2\omega)(\tau_1)-(\mathcal{F}_2\omega)(\tau_2)|\to 0,\; \text{as}\; |\tau_1-\tau_2|\to 0,\;\text{for each}\; \omega\in \mathcal{S}.$$
\noindent This proves $\mathcal{F}_2(\mathcal{S})$ is equicontinious set in $\Omega$. Since $\mathcal{F}_2(\mathcal{S})$ is uniformly bounded and  equicontinious set in $\Omega$, by Ascoli-Arzela theorem  $\mathcal{F}_2 $ is completely continuous.

\noindent Step 3) Let any $\eta\in \mathcal{S}$. For $\omega\in \Omega$, consider the operator equation $ \omega=\mathcal{F}_2\omega \mathcal{F}_2\eta$. Our aim is to prove that $\omega\in \mathcal{S}$.

\noindent 
Using hypothesis (H1) and condition \eqref{HYe4.9}, we have
\begin{align*}
|\omega(\tau)|&=|(\mathcal{F}_1\omega)(\tau)||(\mathcal{F}_2\eta)(\tau)|\\
&\leq |f(\tau,\omega(\tau))||(\mathcal{F}_2\eta)(\tau))|\\
&\leq\left\lbrace |f(\tau,\omega(\tau))-f(\tau,0))|+|f(\tau,0)|\right\rbrace \left( \left| \frac{\omega_0}{f(0,\omega_0)}\right|  +\left[ 1-\alpha+\frac{T^\alpha}{1-\alpha}\right]\frac{\|h\|}{B(\alpha)}\right) \\
&\leq \left\lbrace L|\omega(\tau)|+M_f\right\rbrace \left( \left| \frac{\omega_0}{f(0,\omega_0)}\right|  +\left[ 1-\alpha+\frac{T^\alpha}{1-\alpha}\right]\frac{\|h\|}{B(\alpha)}\right),\;\tau\in J.
\end{align*}
This gives,
\begin{equation*}
|\omega(\tau)|\leq \dfrac {M_fL\left( \left| \dfrac{\omega_0}{f(0,\omega_0)}\right|  +\left[ 1-\alpha+\frac{T^\alpha}{1-\alpha}\right]\dfrac{\|h\|}{B(\alpha)}\right) }{1-L\left( \left| \dfrac{\omega_0}{f(0,\omega_0)}\right|  +\left[ 1-\alpha+\frac{T^\alpha}{1-\alpha}\right]\dfrac{\|h\|}{B(\alpha)}\right)}=R,\;\tau\in J.
\end{equation*}
Therefore,
$$\|\omega\|\leq R. $$
This proves $\omega\in S$.

\noindent Step 4) The constants $\alpha$ and $M$ of Lemma \ref{LM2.3} corresponding to the operators  $\mathcal{F}_1$ and $\mathcal{F}_2$ defined in equations \eqref{e4.6} and \eqref{e4.7} respectively are
$$\alpha=L\;\text{and}\; M=\left| \frac{\omega_0}{f(0,\omega_0)}\right|  +\left[ 1-\alpha+\frac{T^\alpha}{1-\alpha}\right]\frac{\|h\|}{B(\alpha)}.$$
By condition \eqref{b4.4}, it follows that
$$\alpha M=L \left( \left| \frac{\omega_0}{f(0,\omega_0)}\right|  +\left[ 1-\alpha+\frac{T^\alpha}{1-\alpha}\right]\frac{\|h\|}{B(\alpha)}\right) <1. $$
 
 \noindent From steps 1 to 4, it follows that  all the conditions of Lemma \ref{LM2.3} are satisfied. Therefore by applying it, the operator equation $$\omega=\mathcal{F}_1\omega \mathcal{F}_2\omega$$ has a fixed point in $\mathcal{S}$, which is a solution of ABC--hybrid FDEs \eqref{He1.1}--\eqref{He1.2}. This completes the proof of the Theorem.
\end{proof}
 \section{ABC-hybrid fractional differential Inequalities }
 \begin{theorem}\label{Hythm4.1}
 Let $f\in C\left( J\times\mathbb{R}, \mathbb{R}\backslash\left\lbrace 0\right\rbrace\right)$,\;$g\in  C(J\times\mathbb{R},\mathbb{R})$ and for each $\tau\in J$ the mapping  $\omega\to\dfrac{\omega}{f(\tau,\omega)}$ is increasing a.e. on $\mathbb{R}$. Let $v,w\in C(J)$  are such that  $$^{ABC}{_0\mathcal{D}}_{\tau}^\alpha\; \left( \dfrac{v(\cdot)}{f((\cdot),v(\cdot))}\right)\; ,\;^{ABC}{_0\mathcal{D}}_{\tau}^\alpha \left( \dfrac{w(\cdot)}{f((\cdot),w(\cdot))}\right)\in C(J) $$ and  satisfies the ABC-hybrid fractional differential inequalities, 
 \begin{itemize}
 \item[\normalfont({i})] $^{ABC}{_0\mathcal{D}^\alpha_\tau}\left[ \dfrac{v(\tau)}{f(\tau,v(\tau))}\right]\leq g(\tau,v(\tau)), \;\text{a.e.}\;\tau\in J,$
 \item[\normalfont({ii})] $^{ABC}{_0\mathcal{D}^\alpha_\tau}\left[ \dfrac{w(\tau)}{f(\tau,w(\tau))}\right]\geq g(\tau,w(\tau)), \;\text{a.e.}\;\tau\in J, $
 \end{itemize}
 where one of above inequality being is strict.
 
 \noindent Then $v(0)<w(0)$ implies $$v(\tau)<w(\tau),\;\tau\in J.$$
 \end{theorem}
  \begin{proof}
 Suppose that the conclusion of the theorem does not holds.  Since $v,w\in C(J)$ there exits  $\tau_0\in J$ such that $$v(\tau_0)=w(\tau_0)\; \text{and}\; v(\tau)< w(\tau)\;\text{for all} \; \tau\in[0,\tau_0).$$
 Then 
   $$\dfrac{v(\tau_0)}{f(\tau_0,v(\tau_0))}= \dfrac{w(\tau_0)}{f(\tau_0,w(\tau_0))}$$ 
   and using increasing property of the mapping $\omega\to\dfrac{\omega}{f(\tau,\omega)}$, we have
    $$\dfrac{v(\tau)}{f(\tau,v(\tau))}\leq \dfrac{w(\tau)}{f(\tau,w(\tau))}, \tau\in [0,\tau_0).$$
   Let $V(\tau)=\dfrac{v(\tau)}{f(\tau,v(\tau))},\; W(\tau)=\dfrac{w(\tau)}{f(\tau,w(\tau))},\; \tau\in J$. Define $m(\tau)=V(\tau)-W(\tau),\; \tau\in J$.
Then  $m,\;^{ABC}{_0\mathcal{D}^\alpha_\tau}m\in C(J)$. Further, $\tau_0\in J$ is such that
 $$m(\tau_0)=0\; \text{and}\; m(\tau)\leq 0\;\text{for all} \; \tau\in[0,\tau_0).$$ Since $m$ satisfies all assumptions of Lemma \ref{Lm2.4}, we get,  $^{ABC}{_0\mathcal{D}}_{\tau}^\alpha m(\tau_0)\geq 0$. 
 
 \noindent This gives $$^{ABC}{_0\mathcal{D}}_{\tau}^\alpha V(\tau_0)\geq\; ^{ABC}{_0\mathcal{D}}_{\tau}^\alpha W(\tau_0).$$
 
 \noindent Suppose that the inequality (i)  is strict. Then we get
 $$g\left( \tau_0, v(\tau_0)\right)>\; ^{ABC}{_0\mathcal{D}}_{\tau}^\alpha V(\tau_0)\geq\; ^{ABC}{_0\mathcal{D}}_{\tau}^\alpha W(\tau_0) \geq g\left( \tau_0, w(\tau_0)\right),$$
 
 \noindent which is  contradiction to  $v(\tau_0)=w(\tau_0)$. Therefore we must have $$v(\tau)<w(\tau),\;\text{for all}\; \tau\in J.$$
 This completes the proof of theorem. 
  \end{proof}

  \begin{theorem}\label{Hythm4.2}
Assume  that the conditions of Theorem \ref{Hythm4.1} holds with nonstrict inequalities {\normalfont(i)} and {\normalfont(ii)}. Suppose that
\begin{align}\label{c5.1}
 g(\tau,\omega)-g(\tau,\eta)\leq& L\left( \dfrac{\omega}{f(\tau,\omega)}-\dfrac{\eta}{f(\tau,\eta)}\right),\text{for all }\;\tau\in J;\; \omega,\eta\in \mathbb{R}\;\text{with}\;\omega\geq \eta,\;
 0<L<\dfrac{B(\alpha)}{1-\alpha}
\end{align}
Then $v(0)\leq w(0)$ implies
$$v(\tau)\leq w(\tau),\; \text{for all}\;\tau\in J.$$
  \end{theorem}
  \begin{proof}
  For any fix $\epsilon>0$, we define
  \begin{equation}\label{eq3.3}
 \dfrac{w_\epsilon(\tau)}{f(\tau,w_\epsilon(\tau))}=\dfrac{w(\tau)}{f(\tau,w(\tau))}+\epsilon\mathbb{E}_{\alpha}(\tau^\alpha),\; \tau\in J.
  \end{equation}
  This gives, for $\tau=0$ $$\dfrac{w_\epsilon(0)}{f(0,w_\epsilon(0))}=\dfrac{w(0)}{f(0,w(0))}+\epsilon> \dfrac{w(0)}{f(0,w(0))}
  $$
  Again, using the Lipschitz condition on $g$ and Eq.\eqref{eq3.3},  we have
$$g(\tau,w_\epsilon(\tau)) -g(\tau,w(\tau))\leq L\left( \dfrac{w_\epsilon(\tau)}{f(\tau,w_\epsilon(\tau))}-\dfrac{w(\tau)}{f(\tau,w(\tau))}\right) =L\epsilon\mathbb{E}_{\alpha}(\tau^\alpha),\;\tau\in J.$$
  Using  condition on $L$, from above inequality we obtain, 
  \begin{equation}\label{e3.4}
  g(\tau,w(\tau))\geq g(\tau,w_\epsilon(\tau))-L\epsilon\mathbb{E}_{\alpha}(\tau^\alpha)>g(\tau,w_\epsilon(\tau))-\dfrac{B(\alpha)}{1-\alpha}\epsilon\mathbb{E}_{\alpha}(\tau^\alpha),\; \tau\in J.
  \end{equation}
In the proof of Theorem 3.6 \cite{KKSS} it is produced that,
    \begin{equation}\label{e3.7}
  ^{ABC}{_0\mathcal{D}}_\tau^\alpha\left( \mathbb{E}_{\alpha}(\tau^\alpha)\right) \geq \dfrac{B(\alpha)}{1-\alpha}\mathbb{E}_{\alpha}(\tau^\alpha),\tau\in J.
  \end{equation}
  Since $^{ABC}{_0\mathcal{D}}_{\tau}^\alpha \dfrac{w(\cdot)}{f(\cdot,w(\cdot))},\;^{ABC}{_0\mathcal{D}}_{\tau}^\alpha\mathbb{E}_{\alpha}\in C(J),$ we have $^{ABC}{_0\mathcal{D}}_{\tau}^\alpha \dfrac{w_\epsilon(\cdot)}{f(\cdot,w_\epsilon(\cdot))}\in C(J)$. Thus using the inequalities (ii), \eqref{e3.4} and \eqref{e3.7}, for any $\tau\in J$ we have
  \begin{align*}
  ^{ABC}{_0\mathcal{D}}_\tau^\alpha \left( \dfrac{w_\epsilon(\tau)}{f(\tau,w_\epsilon(\tau))}\right) &=\;^{ABC}{_0\mathcal{D}}_\tau^\alpha \left[ \dfrac{w(\tau)}{f(\tau,w(\tau))}+\epsilon\mathbb{E}_{\alpha}(\tau^\alpha)\right]\\
  &=\;^{ABC}{_0\mathcal{D}}_\tau^\alpha  \left( \dfrac{w(\tau)}{f(\tau,w(\tau))}\right) +\epsilon\;^{ABC}{_0\mathcal{D}}_\tau^\alpha\mathbb{E}_{\alpha}(\tau^\alpha)\\
  &\geq g(\tau,w(\tau)) +\epsilon\dfrac{B(\alpha)}{1-\alpha}\mathbb{E}_{\alpha}(\tau^\alpha)\\
  &> g(\tau,w_\epsilon(\tau))-\dfrac{B(\alpha)}{1-\alpha}\epsilon\mathbb{E}_{\alpha}(\tau^\alpha)+\epsilon\dfrac{B(\alpha)}{1-\alpha}\mathbb{E}_{\alpha}(\tau^\alpha)\\
  & =g(\tau,w_\epsilon(\tau))
  \end{align*}
  Therefore, 
  $$^{ABC}{_0\mathcal{D}}_{\tau}^\alpha \left( \dfrac{w_\epsilon(\tau)}{f(\tau,w_\epsilon(\tau))}\right)>  g\left( t, w_\epsilon(\tau)\right),\; \tau\in J.$$
  Since $v(0)<w_\epsilon(0)$, by application of Theorem  \ref{Hythm4.1} with $w(\tau)=w_\epsilon(\tau)$, for each $\epsilon>0$ we have
  $$v(\tau)<w_\epsilon(\tau),\; \tau\in J. $$
  Taking limit as $\epsilon\to 0$, in the above inequality and utilizing equ. \eqref{eq3.3} we obtain
  $$ v(\tau)\leq w(\tau),\; \tau\in J.$$
\end{proof}
  \section{Existence of Maximal and Minimal solutions}
  In this section, we shall prove the existence of maximal and minimal solutions for the ABC--hybrid FDEs \eqref{He1.1}-\eqref{He1.2} on $J $.
  \begin{definition}
A solution $ r$ of the {\normalfont ABC--hybrid FDEs} \eqref{He1.1}-\eqref{He1.2} is said to be maximal if for any other solution $\omega$ to the {\normalfont ABC--hybrid FDEs} \eqref{He1.1}-\eqref{He1.2} one has
$\omega(\tau)\leq \omega(\tau)$ for all $\tau\in J$. Again, a solution $\rho$ of the {\normalfont ABC--hybrid FDEs} \eqref{He1.1}-\eqref{He1.2} is said to be minimal if $\rho(\tau)\leq \omega(\tau)$ for all $\tau \in J$, where $\omega$ is any
solution of the {\normalfont ABC--hybrid FDEs} \eqref{He1.1}-\eqref{He1.2} on $ J$.
\end{definition}

We give the proof only for the existence of maximal solution of the ABC--hybrid FDEs \eqref{He1.1}-\eqref{He1.2}, as the proof of existence of minimal solution one can complete on similar lines. Given an arbitrary small real number $\epsilon>0$, consider the following  ABC--hybrid FDEs
 \begin{align}
 ^{ABC}{_0\mathcal{D}^\alpha_\tau}\left[ \dfrac{\omega(\tau)}{f(\tau,\omega(\tau))}\right] &=g(\tau,\omega(\tau))+\epsilon, \;\text{a.e.}\;\tau\in J,\label{Hye6.1}\\
 \omega(0)&=\omega_0+\epsilon\label{Hye6.2},
 \end{align}
 where $g\in \mathcal{C}$ is such that $g(0,\omega_0+\epsilon)=0$.
 \begin{theorem}\label{Hyth6.1}
Assume that the hypotheses {\normalfont(H1)-(H2)} and the condition \eqref{b4.4}holds. Then  for every small $\epsilon>0$, the {\normalfont ABC--hybrid FDEs} \eqref{Hye6.1}--\eqref{Hye6.2} possesses a solution on $J$.
 \end{theorem}
 \begin{proof}
 By hypothesis,
 $$L\left( \left| \frac{\omega_0}{f(0,\omega_0)}\right|  +\left[ 1-\alpha+\frac{T^\alpha}{1-\alpha}\right]\frac{\|h\|}{B(\alpha)}\right) <1$$
 Then we can find $\epsilon_0>0$ such that
 \begin{equation*}
    L\left(\left| \frac{\omega_0+\epsilon}{f(0,\omega_0+\epsilon)}\right|  +\left[ 1-\alpha+\frac{T^\alpha}{1-\alpha}\right]\frac{\|h\|+\epsilon}{B(\alpha)}\right)<1,\;\text{for}\; 0<\epsilon\leq \epsilon_0.
    \end{equation*}
    
    \noindent Following simillar steps as in the proof of Theorem \ref{Hythm3.2}, one can complete the remaining part of the proof.
  \end{proof}
 \begin{theorem}\label{Hythm6.2}
  Assume that the hypotheses {\normalfont(H1)-(H2)} and the condition \eqref{b4.4}holds. Then  for each small $\epsilon>0$, the {\normalfont ABC--hybrid FDEs} \eqref{He1.1}--\eqref{He1.2} possesses a maximal solution on $J$.
 \end{theorem}
\begin{proof}
Let $\left\lbrace \epsilon_n\right\rbrace_{n=0}^\infty$ be a decreasing sequence of positive numbers converging to $0$ where  $\epsilon_0$ is such that,
\begin{equation*}
    L\left(\left| \frac{\omega_0+\epsilon_0}{f(0,\omega_0+\epsilon_0)}\right|  +\left[ 1-\alpha+\frac{T^\alpha}{1-\alpha}\right]\frac{\|h\|+\epsilon_0}{B(\alpha)}\right)<1.
    \end{equation*} 
    Using $\epsilon_n\leq\epsilon_0,\;n\in\mathbb{N}\cup\left\lbrace 0 \right\rbrace $, it is easy to verify that
\begin{equation*}
 L\left(\left| \frac{\omega_0+\epsilon_n}{f(0,\omega_0+\epsilon_n)}\right|  +\left[ 1-\alpha+\frac{T^\alpha}{1-\alpha}\right]\frac{\|h\|+\epsilon_n}{B(\alpha)}\right)<1,\;\text{for all} \;n\in N\cup\left\lbrace 0\right\rbrace.
\end{equation*}
\noindent Due to above condition,  by Theorem \ref{Hyth6.1}, for each $n\in N\cup \left\lbrace 0 \right\rbrace$,   ABC--hybrid FDEs 
\begin{align}
^{ABC}{_0\mathcal{D}^\alpha_\tau}\left[ \dfrac{\omega(\tau)}{f(\tau,\omega(\tau))}\right] &=g(\tau,\omega(\tau))+\epsilon_n, \;\text{a.e.}\;\tau\in J,\label{He5.3}\\
\omega(0)&=\omega_0+\epsilon_n,\label{He5.4}
\end{align}
 has a solution, say $\omega(\tau,\epsilon_n)$, hence we get
 \begin{align}
 ^{ABC}{_0\mathcal{D}^\alpha_\tau}\left[ \dfrac{\omega(\tau,\epsilon_n)}{f(\tau,\omega(\tau,\epsilon_n))}\right] &=g(\tau,\omega(\tau,\epsilon_n))+\epsilon_n> g(\tau,\omega(\tau,\epsilon_n))\;\text{a.e.}\;\tau\in J,\label{He5.31}\\
 \omega(0,\epsilon_n)&=\omega_0+\epsilon_n,\label{He5.41}
 \end{align}
 The equivalent integral equation of above ABC--hybrid FDEs is 
 \begin{align}\label{e6.7}
\omega(\tau,\epsilon_n)=f(\tau,\omega(\tau,\epsilon_n))&\left[ \dfrac{\omega_0}{f(0,\omega_0+\epsilon_n)}+\dfrac{1-\alpha}{B(\alpha)}  g(\tau,\omega(\tau,\epsilon_n))\right.\nonumber\\
&\left.+\dfrac{\alpha}{B(\alpha)(1-\alpha)}\int_{0}^{\tau}(\tau-\sigma)^{\alpha-1}g(\sigma,\omega(\sigma,\epsilon_n)+\epsilon_n)d\sigma \right]
 \end{align}
Let  $u$ be any solution of ABC--hybrid FDEs \eqref{He1.1}--\eqref{He1.2}, hence we get
\begin{align}
^{ABC}{_0\mathcal{D}^\alpha_\tau}\left[ \dfrac{u(\tau)}{f(\tau,u(\tau))}\right] &=g(\tau,u(\tau)), \;\text{a.e.}\;\tau\in J,\label{He6.7}\\
u(0)&=\omega_0,\label{He6.8}
\end{align}
 Noting that,  $ \omega(0,\epsilon_n)<u(0)$ for all $n\in\mathbb{N}\cup\left\lbrace 0\right\rbrace $. Therefore using comparison Theorem \ref{Hythm4.1}, we have
  \begin{equation}\label{e5.6}
  u(\tau)<\omega(\tau,\epsilon_n),\tau\in J,\; n\in\mathbb{N}\cup\left\lbrace 0 \right\rbrace. 
  \end{equation} 
Let $\omega(\tau,\epsilon_m),\;\omega(\tau,\epsilon_n)$ be the solutions of ABC--hybrid FDEs \eqref{He5.3}-\eqref{He5.4} corresponding to the $m^{th},\;n^{th}$ term of the sequence  $\left\lbrace \epsilon_n\right\rbrace_{n=0}^\infty$, with $m>n$. Therefore we have,
\begin{align*}
\omega(0,\epsilon_m)=\omega_0+\epsilon_m&<\omega_0+\epsilon_n=\omega(0,\epsilon_n)\\
 ^{ABC}{_0\mathcal{D}^\alpha_\tau}\left[ \dfrac{\omega(\tau,\epsilon_n)}{f(\tau,\omega(\tau,\epsilon_n))}\right] &=g(\tau,\omega(\tau,\epsilon_n))+\epsilon_n\\
  ^{ABC}{_0\mathcal{D}^\alpha_\tau}\left[ \dfrac{\omega(\tau,\epsilon_m)}{f(\tau,\omega(\tau,\epsilon_m))}\right] &\leq g(\tau,\omega(\tau,\epsilon_m))+\epsilon_n
\end{align*} $$ $$
Applying Lemma \ref{Hythm4.1} to the above set of inequalities, we get 
$$\omega(\tau,\epsilon_m) <\omega(\tau,\epsilon_n).$$ This verifies that $\omega(\tau,\epsilon_m)$ decreasing sequence bounded bellow by any solution of ABC--hybrid FDEs \eqref{He1.1}--\eqref{He1.2}. Therefore  $\omega(\tau)=\lim\limits_{n\to\infty}\omega(\tau,\epsilon_n)$ exists on $J$. We show that  this converges is unoform on $J$. Therefore, it is enough to prove that the sequence $ \left\lbrace \omega(\tau, \epsilon_n) \right\rbrace$ is
 equicontinuous in $C(J, R)$. Let $ \tau_1, \tau_2 \in J$ with $\tau_1 < \tau_2 $ be arbitrary. Then,
 \begin{align}\label{ie6.6}
 &|\omega(\tau_1,\epsilon_n)-\omega(\tau_2,\epsilon_n)|\leq \left( \dfrac{|\omega_0+\epsilon_n|}{|f(0,\omega_0+\epsilon_n)|}+\dfrac{1-\alpha}{B(\alpha)}\epsilon_n\right) |f(\tau_1,\omega(\tau_1,\epsilon_n))-f(\tau_2,\omega(\tau_2,\epsilon_n))|\nonumber\\
 &\;+\dfrac{1-\alpha}{B(\alpha)}|f(\tau_1,\omega(\tau_1,\epsilon_n))g(\tau_1,\omega(\tau_1,\epsilon_n))-f(\tau_2,\omega(\tau_2,\epsilon_n))g(\tau_2,\omega(\tau_2,\epsilon_n))|\nonumber\\
 &\;+\dfrac{\alpha}{B(\alpha)(1-\alpha)}\left| f(\tau_1,\omega(\tau_1,\epsilon_n))\int_{0}^{\tau_1}(\tau_1-\sigma)^{\alpha-1} \left\lbrace g(\sigma,\omega(\sigma,\epsilon_n))+\epsilon_n\right\rbrace  d\sigma\right.\nonumber\\
&\left.\qquad-f(\tau_2,\omega(\tau_2,\epsilon_n))\int_{0}^{\tau_2}(\tau_2-\sigma)^{\alpha-1}\left\lbrace  g(\sigma,\omega(\sigma,\epsilon_n))+\epsilon_n\right\rbrace d\sigma\right|.
 \end{align}
 Since $ f,g$  are continuous on compact set $J\times[-R,R]$, they are  uniformly continuous there. Hence, for each $n\in\mathbb{N}$,
 \begin{align}
 |f(\tau_1,\omega(\tau_1,\epsilon_n))&-f(\tau_2,\omega(\tau_2,\epsilon_n))|\to 0,\;\text{as}\; |\tau_1-\tau_2|\to 0\label{ie6.11}\\
 |f(\tau_1,\omega(\tau_1,\epsilon_n))g(\tau_1,\omega(\tau_1,\epsilon_n))&-f(\tau_2,\omega(\tau_2,\epsilon_n))g(\tau_2,\omega(\tau_2,\epsilon_n))|\to 0,\;\text{as}\; |\tau_1-\tau_2|\to 0\label{ie6.12}
 \end{align}
 Let $F=\underset{(\tau,\omega)\in J\times[-R,R]}{\sup f(\tau,\omega)}$. We find
  \begin{align*}
 & \left| f(\tau_1,\omega(\tau_1,\epsilon_n))\int_{0}^{\tau_1}(\tau_1-\sigma)^{\alpha-1} \left\lbrace g(\sigma,\omega(\sigma,\epsilon_n))+\epsilon_n\right\rbrace  d\sigma\right.\nonumber\\
  &\left.\qquad-f(\tau_2,\omega(\tau_2,\epsilon_n))\int_{0}^{\tau_2}(\tau_2-\sigma)^{\alpha-1} \left\lbrace g(\sigma,\omega(\sigma,\epsilon_n))+\epsilon_n\right\rbrace d\sigma\right|\\
  &= \left|\left(  f(\tau_1,\omega(\tau_1,\epsilon_n))\int_{0}^{\tau_1}(\tau_1-\sigma)^{\alpha-1} \left\lbrace g(\sigma,\omega(\sigma,\epsilon_n))+\epsilon_n\right\rbrace  d\sigma\right.\right.\nonumber\\
  &\left.\left.\qquad-f(\tau_1,\omega(\tau_1,\epsilon_n))\int_{0}^{\tau_1}(\tau_1-\sigma)^{\alpha-1} \left\lbrace g(\sigma,\omega(\sigma,\epsilon_n))+\epsilon_n\right\rbrace d\sigma\right) \right|\\
  & +\left|\left(  f(\tau_1,\omega(\tau_1,\epsilon_n))\int_{0}^{\tau_2}(\tau_2-\sigma)^{\alpha-1} \left\lbrace g(\sigma,\omega(\sigma,\epsilon_n))+\epsilon_n\right\rbrace  d\sigma\right.\right.\nonumber\\
&\left.\left.\qquad-f(\tau_2,\omega(\tau_2,\epsilon_n))\int_{0}^{\tau_2}(\tau_2-\sigma)^{\alpha-1} \left\lbrace g(\sigma,\omega(\sigma,\epsilon_n))+\epsilon_n\right\rbrace d\sigma\right) \right|\\
&\leq F \left\lbrace \int_{0}^{\tau_1}\left\lbrace (\tau_1-\sigma)^{\alpha-1}-(\tau_2-\sigma)^{\alpha-1}\right\rbrace   |g(\sigma,\omega(\sigma,\epsilon_n))+\epsilon_n|  d\sigma\right.\nonumber\\
  &\left.\qquad
\qquad +\int_{\tau_1}^{\tau_2}(\tau_2-\sigma)^{\alpha-1}  |g(\sigma,\omega(\sigma,\epsilon_n))+\epsilon_n| d\sigma\right\rbrace\\
&\quad+|f(\tau_1,\omega(\tau_1,\epsilon_n))-f(\tau_2,\omega(\tau_2,\epsilon_n))|\int_{0}^{\tau_2}(\tau_2-\sigma)^{\alpha-1} \left\lbrace |g(\sigma,\omega(\sigma,\epsilon_n))+\epsilon_n|\right\rbrace  d\sigma\\
&\leq 2F(\|h\|+\epsilon)(\tau_2-\tau_1)^\alpha+\tau_2^\alpha(\|h\|+\epsilon)|f(\tau_1,\omega(\tau_1,\epsilon_n))-f(\tau_2,\omega(\tau_2,\epsilon_n))|
  \end{align*}
 This shows that,
 \begin{align}\label{ie6.13}
 &\left| f(\tau_1,\omega(\tau_1,\epsilon_n))\int_{0}^{\tau_1}(\tau_1-\sigma)^{\alpha-1} \left\lbrace g(\sigma,\omega(\sigma,\epsilon_n))+\epsilon_n\right\rbrace  d\sigma\right.\nonumber\\
     &\left.\qquad-f(\tau_2,\omega(\tau_2,\epsilon_n))\int_{0}^{\tau_2}(\tau_2-\sigma)^{\alpha-1} \left\lbrace g(\sigma,\omega(\sigma,\epsilon_n))+\epsilon_n\right\rbrace d\sigma\right|  \to 0,\text{as}\,|\tau_2-\tau_1|\to 0.
\end{align}  

\noindent Using inequalities \eqref{ie6.11}, \eqref{ie6.12} and \eqref{ie6.13}, we conclude from \eqref{ie6.6} that
$$|\omega(\tau_1,\epsilon_n)-\omega(\tau_2,\epsilon_n)|\to 0,,\text{as}\,|\tau_2-\tau_1|\to 0, $$  This shows that  $\left\lbrace \omega(\tau,\epsilon_n) \right\rbrace $ converges uniformly to $\omega(\tau)$ as $ n\to\infty$. Hence taking limit as $n\to\infty$ of equation Equ.\eqref{e6.7}, we get
$$\omega(\tau)=f(\tau,\omega(\tau))\left[ \dfrac{\omega_0}{f(0,\omega_0)}+\dfrac{1-\alpha}{B(\alpha)}  g(\tau,\omega(\tau))+\dfrac{\alpha}{B(\alpha)(1-\alpha)}\int_{0}^{\tau}(\tau-\sigma)^{\alpha-1}g(\sigma,\omega(\sigma))d\sigma \right]. $$
Thus $\omega(\tau)$ is a solution of ABC--hybrid FDEs \eqref{He1.1}--\eqref{He1.2}. Taking limit as $n\to\infty$ of inequality \eqref{e5.6}, we get $u(\tau)<\omega(\tau),\;\tau\in J$. Hence ABC--hybrid FDEs \eqref{He1.1}--\eqref{He1.2} has a maximal solution.
\end{proof}
\section{Comparison Results}
\begin{theorem}\label{Th7.1}
Suppose the hypotheses {\normalfont(H1)-(H2) }and condition \eqref{b4.4} hold. Also assume that the function $g$ satisfies the condition \eqref{c5.1}.  If there exists a function $u\in AC(J,\mathbb{R})$ such that
\begin{align}
 ^{ABC}{_0\mathcal{D}^\alpha_\tau}\left[ \dfrac{u(\tau)}{f(\tau,u(\tau))}\right] &\leq g(\tau,u(\tau)), \;\text{a.e.}\;\tau\in J,\label{He6.1}\\
 u(0)&\leq \omega_0,\label{He6.2}
 \end{align}
 Then
 $$u(\tau)\leq \omega(\tau),\;\tau\in J,$$
 where $r$ is maximal solution of the  {\normalfont ABC--hybrid FDEs} \eqref{He1.1}--\eqref{He1.2}.
\end{theorem}
\begin{proof}
Let $\epsilon>0$ be arbitrary small. Then by Theorem \ref{Hythm6.2},\; $\omega(\tau,\epsilon)$ is a solution of the ABC--hybrid FDEs \eqref{Hye6.1}--\eqref{Hye6.2}. Therefore  
\begin{align*}
 ^{ABC}{_0\mathcal{D}^\alpha_\tau}\left[ \dfrac{\omega(\tau,\epsilon)}{f(\tau,\omega(\tau,\epsilon))}\right] &= g(\tau,\omega(\tau,\epsilon))+\epsilon, \;\text{a.e.}\;\tau\in J,\\
 \omega(0,\epsilon)= \omega_0+\epsilon
\end{align*}
This gives,
\begin{align}
 ^{ABC}{_0\mathcal{D}^\alpha_\tau}\left[ \dfrac{\omega(\tau,\epsilon)}{f(\tau,\omega(\tau,\epsilon))}\right] &> g(\tau,\omega(\tau,\epsilon)), \;\text{a.e.}\;\tau\in J,\\
 \omega(0,\epsilon)> \omega_0
\end{align}
Therefore $u$ is lower solution and $\omega(\tau,\epsilon)$ is upper solution of  $$^{ABC}{_0\mathcal{D}^\alpha_\tau}\left[ \dfrac{\omega(\tau)}{f(\tau,\omega(\tau))}\right] =g(\tau,\omega(\tau)).$$ 
Further,
$$
u(0)\leq \omega_0<\omega_0+\epsilon=\omega(0,\epsilon).
$$
By applying Theorem \ref{Hythm4.1}, we obtain
$$ u(\tau)<\omega(\tau,\epsilon), \;\text{for all}\; \tau\in J,$$ 
In limiting case as $\epsilon\to 0$, we get
$$u(\tau)\leq \omega(\tau),\; \tau\in J.$$
\end{proof}

The proof of the following Theorem can be given in similar way as in the case of Theorem \ref{Th7.1}.
\begin{theorem}
Suppose the hypotheses {\normalfont(H1)-(H2) }and condition \eqref{b4.4} hold. Also assume that the function $g$ satisfies the condition \eqref{c5.1}.  If there exists a function $v\in AC(J,\mathbb{R})$ such that
\begin{align}
 ^{ABC}{_0\mathcal{D}^\alpha_\tau}\left[ \dfrac{v(\tau)}{f(\tau,v(\tau))}\right] &\geq g(\tau,v(\tau)), \;\text{a.e.}\;\tau\in J,\\
v(0)&\geq \omega_0,
 \end{align}
 Then
 $$\rho(\tau)\leq v(\tau),\;\tau\in J.$$
 where $\rho$ is minimal solution of the  {\normalfont ABC--hybrid FDEs} \eqref{He1.1}--\eqref{He1.2}.
\end{theorem}

Using Theorem \ref{Th7.1}, we can prove the uniqueness result for ABC--hybrid FDEs \eqref{He1.1}--\eqref{He1.2}. The detail of which is given in the following Theorem.
\begin{theorem}
Suppose the hypotheses {\normalfont(H1)-(H2) }and condition \eqref{b4.4} hold. Also assume that there exists a function $G:J\times \mathbb{R}^+\to \mathbb{R}^+$ satisfying
$$ |g(\tau,\omega)-g(\tau,\eta)|\leq G\left(\tau, \left| \dfrac{\omega}{f(\tau,\omega)}-\dfrac{\eta}{f(\tau,\eta)}\right|\right) ,\text{for all }\;\tau\in J;\; \omega,\eta\in \mathbb{R}.$$
If identically zero function is the only solution of the ABC--hybrid FDEs 
\begin{align}
^{ABC}{_0\mathcal{D}^\alpha_\tau}m(\tau) &=G(\tau,m(\tau)), \;\text{a.e.}\;\tau\in J,\;0<\alpha<1,\label{He7.7}\\
 m(0)&=0,\label{He7.8}
 \end{align}
 
\noindent Then the ABC--hybrid FDEs \eqref{He1.1}--\eqref{He1.2} has a unique solution on J.
\end{theorem}
\begin{proof}
As the required assumptions (H1)-(H2) and the condition \eqref{b4.4} hold, the ABC--hybrid FDEs \eqref{He1.1}--\eqref{He1.2} has a  solution on J. Suppose there exists two solutions $\omega,\eta$ for ABC--hybrid FDEs \eqref{He1.1}--\eqref{He1.2}.

\noindent Define
$$ m(\tau)=\left| \dfrac{\omega}{f(\tau,\omega)}-\dfrac{\eta}{f(\tau,\eta)}\right|,\;\tau\in J.$$
We find by using linearity 
\begin{align}\label{ie7.9}
^{ABC}{_0\mathcal{D}^\alpha_\tau}\left( \dfrac{\omega}{f(\tau,\omega)}-\dfrac{\eta}{f(\tau,\eta)}\right)&= \;^{ABC}{_0\mathcal{D}^\alpha_\tau}\left[ \dfrac{\omega(\tau)}{f(\tau,\omega(\tau))}\right] -\;^{ABC}{_0\mathcal{D}^\alpha_\tau}\left[ \dfrac{\eta(\tau)}{f(\tau,\eta(\tau))}\right]\nonumber\\
&=g(\tau,\omega(\tau))-g(\tau,\eta(\tau))\leq |g(\tau,\omega(\tau))-g(\tau,\eta(\tau))|\nonumber\\
&\leq G\left(\tau, \left| \dfrac{\omega}{f(\tau,\omega)}-\dfrac{\eta}{f(\tau,\eta)}\right|\right)=G(\tau,m(\tau)),\;\tau\in J.
\end{align} 
Again by using the fact that $|f|'\leq|f'|$, we find
\begin{align}\label{ie7.10}
^{ABC}{_0\mathcal{D}^\alpha_\tau}m(\tau)&= \dfrac{B(\alpha)}{1-\alpha}\int_{0}^{\tau}\mathbb{E}_\alpha\left[ -\dfrac{\alpha}{1-\alpha}(\tau-\sigma)^\alpha\right] \left| \dfrac{\omega}{f(\tau,\omega)}-\dfrac{\eta}{f(\tau,\eta)}\right|'d\sigma\nonumber\\
&\leq \dfrac{B(\alpha)}{1-\alpha}\int_{0}^{\tau}\mathbb{E}_\alpha\left[ -\dfrac{\alpha}{1-\alpha}(\tau-\sigma)^\alpha\right] \left| \left( \dfrac{\omega}{f(\tau,\omega)}-\dfrac{\eta}{f(\tau,\eta)}\right)' \right|d\sigma\nonumber\\
&\leq \left| \dfrac{B(\alpha)}{1-\alpha}\int_{0}^{\tau}\mathbb{E}_\alpha\left[ -\dfrac{\alpha}{1-\alpha}(\tau-\sigma)^\alpha\right]  \left( \dfrac{\omega}{f(\tau,\omega)}-\dfrac{\eta}{f(\tau,\eta)}\right)' d\sigma\right|\nonumber\\
&=\;^{ABC}{_0\mathcal{D}^\alpha_\tau}\left( \dfrac{\omega}{f(\tau,\omega)}-\dfrac{\eta}{f(\tau,\eta)}\right)
\end{align}
Combining inequalities \eqref{ie7.9} and \eqref{ie7.10}, we obtain
\begin{equation}\label{eq7.11}
^{ABC}{_0\mathcal{D}^\alpha_\tau}m(\tau)\leq G(\tau,m(\tau)),\; \tau\in J.
\end{equation}
Using definition of $m(\tau)$, we find 
\begin{equation}\label{eq7.12}
m(0)=\left|\dfrac{\omega(0)}{f(0,\omega(0)_)}-\dfrac{\eta(0)}{f(0,\eta(0))}\right|=\left| \dfrac{\omega_0}{f(0,\omega_0)}-\dfrac{\omega_0}{f(0,\omega_0)}\right|=0
\end{equation}
From equations \eqref{eq7.11} and \eqref{eq7.12}, using assumption, we get $m(\tau)=0,\;\tau\in J. $
From which we can easily show that
$$\dfrac{\omega}{f(\tau,\omega)}=\dfrac{\eta}{f(\tau,\eta)},\;\tau\in J. $$
Hence $\omega=\eta$. This proves the uniqueness of solution.
\end{proof}

\section*{Conclusion}
Fractional  integral inequalities and comparison results acquired in the present paper can utilized to analyze the various qualitative and quantitative properties of solutions for a different class of ABC--hybrid FDEs subject to enhanced initial and boundary conditions.

\end{document}